\documentclass[11pt]{article} 
\usepackage[utf8]{inputenc}
\usepackage[T1]{fontenc}
\usepackage[english]{babel}
\usepackage{ulem}

\usepackage{amsmath,amsfonts,amssymb,amsthm,mathrsfs,amsrefs, bbm}
\usepackage{tikz}

\renewcommand{\leq}{\leqslant}
\renewcommand{\geq}{\geqslant}
\usepackage[cmintegrals,libertine]{newtxmath}
\usepackage[cal=euler]{mathalfa}
\usepackage[mono=false]{libertine}
\useosf
\usepackage[a4paper,vmargin={3.5cm,3.5cm},hmargin={2.5cm,2.5cm}]{geometry}
\usepackage[font={small,sf}, labelfont={sf,bf}, margin=1cm]{caption}
\usepackage[pdftex,colorlinks=true]{hyperref}
\usepackage[expansion]{microtype}

\usepackage{stackrel}
\usepackage{soul} 

\theoremstyle{plain}
\newtheorem{theorem}{Theorem}

\newtheorem{proposition}[theorem]{Proposition}
\newtheorem{lemma}[theorem]{Lemma}
\theoremstyle{definition}

\newcommand{\eps}{\varepsilon}

\hypersetup{
    pdftitle    = {}}

\usepackage{overpic, float}

\title{On the minimal diameter of closed hyperbolic surfaces}
\author{\textsc{Thomas Budzinski}\footnote{University of British Columbia. E-mail: \href{mailto:budzinski@math.ubc.ca}{budzinski@math.ubc.ca}.}, \textsc{Nicolas Curien}\footnote{Universit\'e Paris-Saclay and Institut Universitaire de France. E-mail: \href{mailto:nicolas.curien@gmail.com}{nicolas.curien@gmail.com}.} \, and \textsc{Bram Petri}\footnote{Sorbonne Universit\'e. E-mail: \href{mailto:bram.petri@imj-prg.fr}{bram.petri@imj-prg.fr}}}

\begin{document}

\maketitle
\begin{abstract} We prove that the minimal diameter of a hyperbolic compact orientable surface of genus $g$ is asymptotic to $\log g$ as $g \to \infty$. The proof relies on a random construction, which we analyse using lattice point counting theory and the exploration of random trivalent graphs.
\end{abstract}

\section*{Introduction}
When studying the various shapes a closed orientable hyperbolic (i.e. with constant curvature equal to $-1$) surface of a given genus can have, the diameter is a natural geometric invariant to consider. It's interesting in its own right and it also relates to spectral and isoperimetric properties of the given surface. For $g\geq 2$, define
\[ D_g = \min\left\{\mathrm{diam}(X) ;\; X \mbox{ closed orientable hyperbolic surface of genus  }g\right\},\]
where  $ \mathrm{diam}(X)$ is the diameter of a metric space $X$. It is easy to see (using the collar lemma) that one can construct hyperbolic surfaces of a fixed genus $ g \geq 2$ of arbitrarily large diameter. On the other hand, a simple area argument yields that the diameter $\mathrm{diam}(X)$ of a closed orientable hyperbolic surface $X$ of genus $g$ satisfies\footnote{The actual lower bound is $\cosh^{-1}(2g-1)$, but that rolls off the tongue less well.}
\begin{equation}\label{general_lower_bound_diameter}
\mathrm{diam}(X) \geq \log(4g-2),
\end{equation}
so the function $D_{g}$ is asymptotically bounded from below by $\log g$ as $g \to \infty$. The sharpest known lower bound is due to Bavard \cite{Bavard}, which improves on \eqref{general_lower_bound_diameter} by at most an additive constant. Our goal is to prove that this lower bound is asymptotically sharp:
\begin{theorem}\label{thm_main} We have:
\[ \lim_{g\to\infty} \frac{D_g}{\log g} = 1.\]
\end{theorem}

Similar to the case of regular graphs, \emph{random} surfaces are a good source of examples of surfaces with small diameter, large Cheeger constant or large spectral gap of the Laplacian \cite{BM04,Mirzakhani} (arithmetic surfaces have similar properties in this regard \cite{BrooksNumberTheory,Selberg, SarnakXue, Brooks, Clozel}).

The model we use is based on gluing hyperbolic \emph{pairs of pants} along the combinatorics of a random trivalent graph (without twist), see Section \ref{sec:construct}. \emph{Those surfaces will be parametrized by the common length $a \in (0, \infty)$ of the perimeters of the pair of pants we glue together.} One key input will be to explore the neighborhood of a typical point in the underlying trivalent graph \emph{for the hyperbolic distance in the surface} and show that this neighborhood is essentially tree-like up to a few defects. This part of the argument is similar to that of Bollobas \& de la Vega \cite{BFdlV} who proved that a random $3$-regular graph $ \mathcal{G}_{n}$ on $n$ vertices satisfies $\frac{\mathrm{diam}(\mathcal{G}_n)}{\log_{2}(n)} \to 1$ in probability as $n \to \infty$. That last result is the discrete analogue of our Theorem \ref{thm_main} since it is the minimal growth rate that comes out of the argument analogous to the area argument for hyperbolic surfaces (i.e. the ball of radius $r$ for the graph distance has size at most of order $2^r$).  We then use lattice point counting theory to control the volume growth for the hyperbolic metric in a trivalent tree of pants, parametrized by a side length $a \in (0,\infty)$. Finally letting $a \to \infty$, we get our main result.

Note that there are other models of random hyperbolic surfaces, notably the Brooks-Makover construction of random Riemann surfaces  $X_n$ \cite{BM04}, in which $2n$ hyperbolic ideal triangles are glued uniformly at random along their boundary to create a surface of genus $ g \sim \frac{n}{2}$. In a companion paper \cite{BCP19+} we show that they actually satisfy $\mathrm{diam}(X_n) \sim  2\log(g)$ and so they miss the asymptotic lower bound \eqref{general_lower_bound_diameter} by a factor $2$.

\begin{figure}[!h]
 \begin{center}
 \includegraphics{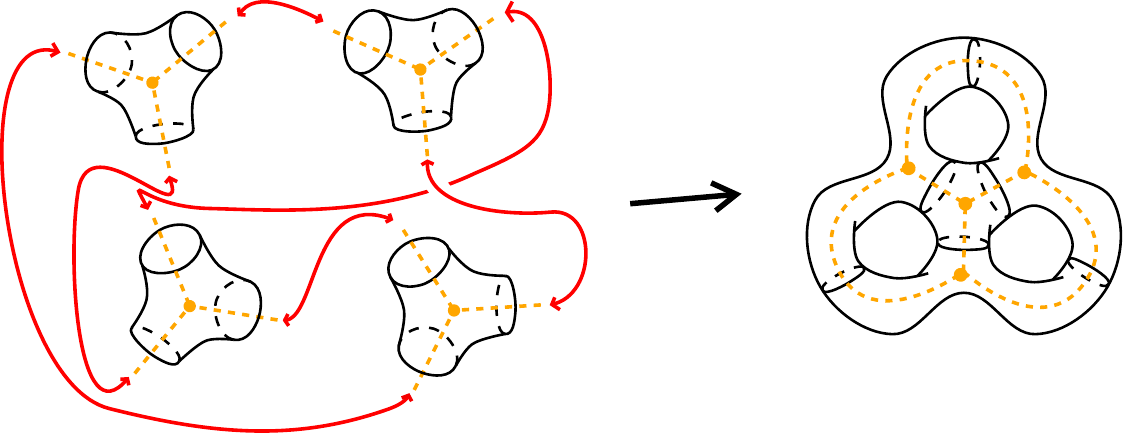}
 \caption{ \label{fig:construction}Construction of the random surfaces used in the proof of Theorem \ref{thm_main}. We start with $2n$ identical copies of a pair of pants with three boundary components of the same length $a \in (0, \infty)$ and pair these boundary components uniformly at random. The paired components get glued without twist. }
 \end{center}
 \end{figure}

\textbf{Acknowledgments:} We thank Maxime Fortier Bourque, Ursula Hamenst\"adt, Pierre Pansu, Hugo Parlier, Fr\'ed\'eric Paulin, and Juan Souto for useful discussions. We acknowledge support from \texttt{ERC 740943} ``GeoBrown''. 

\section{Background material on hyperbolic geometry}
In this section we present the construction of a trivalent tree of pants and analyse its  ``volume growth'' based on lattice point counting theory for convex cocompact groups of isometries in the hyperbolic plane. We then introduce the model of random surface we will use to prove our main theorem.

\subsection{Pants trees and their volume growth}

Recall that a hyperbolic metric on a pair of pants (a surface homeomorphic to a sphere out of which three disjoint open disks have been removed) with totally geodesic boundary is completely determined by the lengths of the three boundary components. Given $a\in (0,\infty)$ we will denote the hyperbolic pair of pants with three boundary components of length $a$ by $P_a$. For later use, we will once and for all fix a basepoint $p_0 \in P_a$. For convenience, we will let $p_0$ be one of the two points at equal distance from all the boundary components of $P_a$.

Furthermore, $T_a$ will denote the hyperbolic surface of infinite volume that is obtained by gluing countably many copies of $P_a$ together in the shape of a trivalent tree. More precisely, if $\mathcal{T}_3$ is an infinite simplicial trivalent tree, we associate to each vertex a copy of $P_a$ and to each edge emanating from this vertex one of the boundary components of that copy $P_a$. For each edge we then glue the two boundary components associated to it together without twist, in such a way that the two copies of the base point $p_0$ (we will call these midpoints) in the corresponding two pairs of pants are on the same side (see Figure \ref{pic_pants_tree} for a sketch).
\begin{figure}[!h]
\begin{center}
\includegraphics[scale=1]{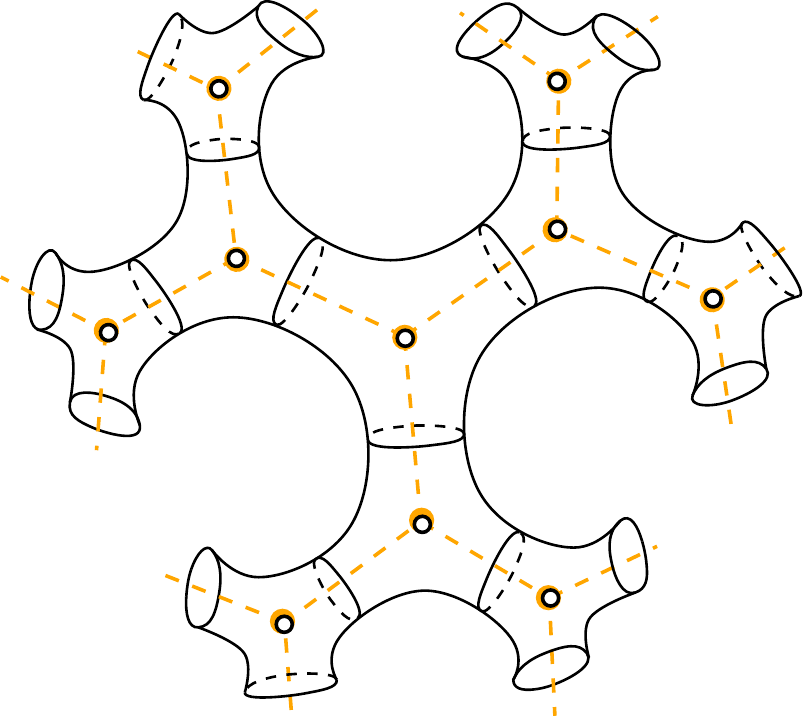}
\end{center}
\caption{Generations $0$, $1$ and $2$ of the infinite trivalent tree and the pants tree together with a base point on each pair of pants.}
\label{pic_pants_tree}
\end{figure}
 
For our application, we will be interested in the ``volume growth'' of these pants trees. To this end, we let $N_a(R)$ denote the number of midpoints at distance at most $R$ from some basepoint $O \in T_a$, that is also a midpoint that we fix once and for all. The growth of this quantity as a function of $R$ is a classical problem and fits in a vast body of literature on lattice point counting (see  \cite[\S 1.3]{GorodnikNevo} for an overview). The result we will use is a direct consequence of work by Patterson \cite{Patterson1, Patterson 2}, Lax-Phillips \cite{LaxPhillips}) and McMullen \cite[Theorem 3.5]{McMullen2}:

\begin{theorem}  \label{thm:counting} For any $a\in (0,\infty)$, there are constants $\mathrm{cst}_a$ and $\delta_a \in (0,1)$ such that 
$$ N_a(R) \sim \mathrm{cst}_a \cdot e^{\delta_a R} \quad \text{as } R\to \infty .$$
Furthermore, $\delta_a \to 1$ as $a\to \infty$.
\end{theorem}

\begin{proof} Let $\Gamma^{\mathrm{hex}}_a$ denote the Fuchsian group generated by the reflections in three non-consecutive sides of a right angled hyperbolic hexagon $H_a$ such that the three non-consecutive sides all have length $a$. We will uniformize so that $0$ is a lift of the point at equal distance from all sides of length $a$ of $H_a$.

Given $R>0$, we set
\[N^{\mathrm{hex}}_a(R) = \#\{ \Gamma_a^{\mathrm{hex}} \cdot 0 \cap B_R(0) \},\]
where $B_R(0)$ denotes the ball of radius $R$ around $0\in \mathbb{H}^2$. Since $\Gamma^{\mathrm{hex}}_a$ is a convex cocompact group, the latter quantity is asymptotic to $\mathrm{cst}_a \cdot e^{\delta_a R}$ as $R\to \infty$, where $\delta_a$ is the critical exponent of $\Gamma_a$ (see \cite{Patterson1, Patterson 2, LaxPhillips} for details).

We claim that
\[N_a(R) = N^{\mathrm{hex}}_a(R) \quad \text{for all } R>0,\]
which would imply the first claim. Indeed, $T_a$ consists of two copies of the trivalent tree of hexagons (or hextree \cite{Ken15}) $C_a = \Gamma^{\mathrm{hex}}_a \cdot H_a$ (see Figure \ref{pic_hextrees}). Let us denote by $C^1$ the copy of $C_a$ that (by construction) contains all the midpoints and by $C^2$ the other copy. 
\begin{figure}[!h]
\begin{center}
\includegraphics[scale=.5]{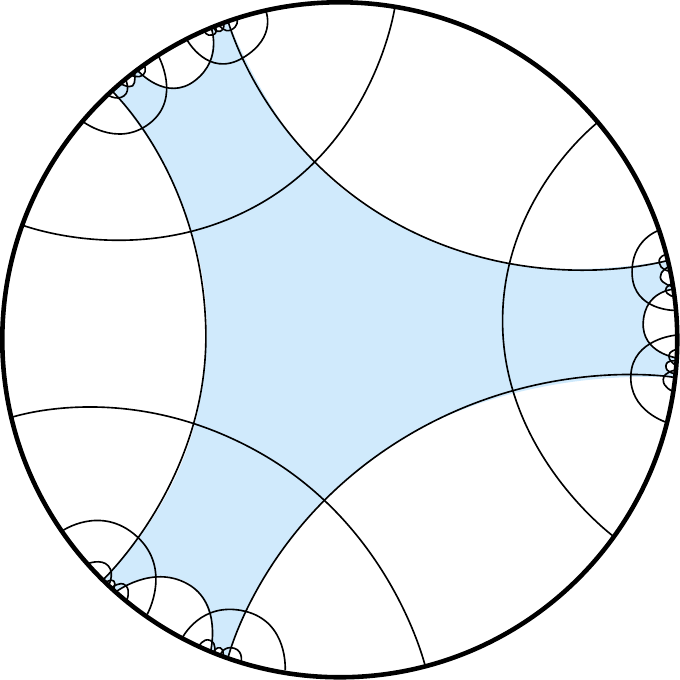}
\includegraphics[scale=.5]{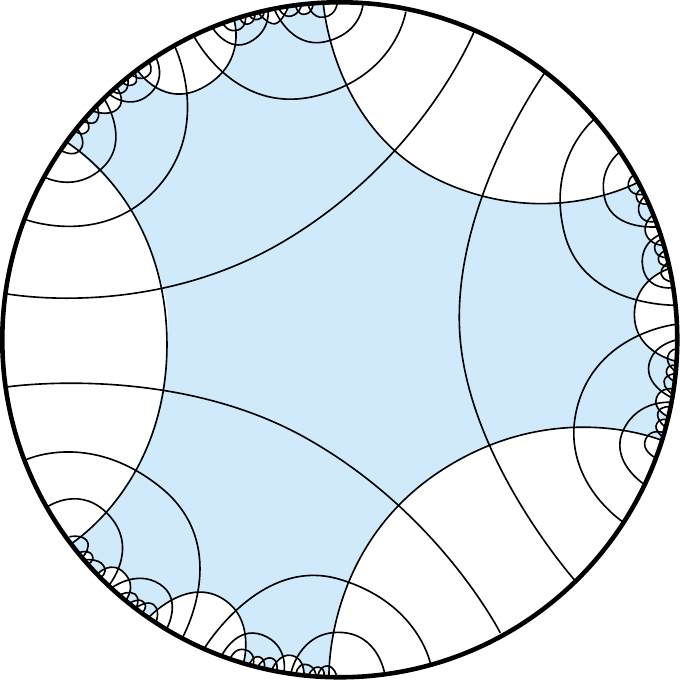}
\includegraphics[scale=.5]{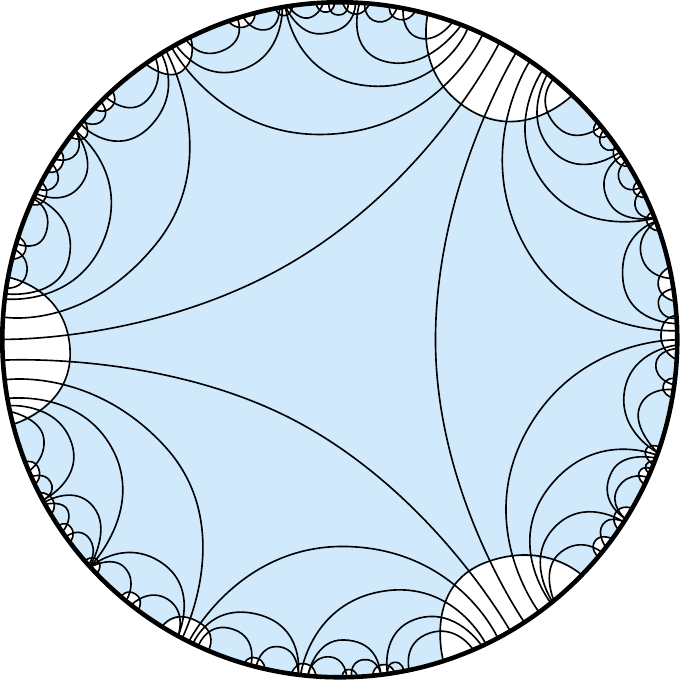}
\end{center}
\caption{$C_a$ for increasing values of $a$.}
\label{pic_hextrees}
\end{figure}

We now claim that all the distance-realizing geodesics between midpoints lie entirely in $C^1$. Suppose there are midpoints $m_1, m_2 \in T_ a$ for which this is not the case. Since $m_1, m_2 \in C^1$, the distance realizing geodesic between them starts and ends in $C^1$ and contains (potentially several) segments that run between two points on the boundary of $C^1$ and $C^2$ and are otherwise contained in $C^2$. By symmetry, we may reflect all these segments into $C^1$ without changing the length of the path. This already means that the distance can be realized by a path in $C^1$. Moreover, since the reflection creates singularities at the boundary between $C^1$ and $C^2$, the resulting path can be made shorter by smoothing the singularities inside $C^1$, making the path lie entirely in $C^1$. In conclusion, distances between midpoints are the same in $T_a$ and $C^1$ and as such, $N_a(R) = N^{\mathrm{hex}}_a(R)$.

The fact that $\delta_a \to 1$ as $a\to\infty$ follows from \cite[Theorem 3.5]{McMullen2}.
\end{proof}

\subsection{The model of random surface and its properties}
\label{sec:construct}

Given an random uniform trivalent graph $\mathcal{G}_n$ on {$2n$} vertices we consider the random surface 
$$S_{a,n}$$ 
obtained as follows. We associate to each vertex of $\mathcal{G}_{n}$ a copy of $P_a$. We then glue these copies along their boundary components as dictated by the edges of $\mathcal{G}_{n}$. The gluing we use is the same as in the construction of $T_a$, i.e. no twist and the distance between midpoints is minimal. As soon as $ \mathcal{G}_n$ is connected (which happens with high probability),  $S_{a,n}$ is an oriented hyperbolic surface of genus $n+1$.

The copies of $P_a$ that $S_{a,n}$ is constructed out of will be called the \emph{vertices} of  $S_{a,n}$ and are also in correspondence with the \emph{midpoints} $m_f$ of the pairs of pants $f$ which are (the images of) the point corresponding to $p_0\in P_a$. Notice for future reference that for fixed $a$, the diameter of the pair of pants $P_a$ is bounded and so any point in $S_{a,n}$ is within bounded distance of the boundary and within bounded distance of a midpoint of a pair of pants.

\section{Bounding the diameter of \texorpdfstring{$S_{a,n}$}{TEXT}}
The main technical result of this paper consists in bounding the diameter of $S_{a,n}$. 
\begin{proposition} \label{prop:diameter}For any $ \varepsilon>0$, with high probability as $n \to \infty$ we have 
 \begin{eqnarray} \mathrm{Diam}( S_{a,n}) \leq \left(\frac{1}{ \delta_a} + \varepsilon\right) \log n.  \label{eq:diameter} \end{eqnarray}
\end{proposition}

Let us first deduce our main result from it:
\begin{proof}[Proof of Theorem \ref{thm_main}] Because of the area argument, all we need to prove is that $\limsup_{g\to\infty} D_g/\log(g) \leq 1 + \varepsilon$ for all $\varepsilon>0$. So, using Theorem \ref{thm:counting} we choose $a$ large enough so that $ \frac{1}{\delta_a} \leq 1 + \varepsilon$. Proposition \ref{prop:diameter} provides us with a sequence of random surfaces $S_{a,n}$ of genus $n+1$ whose diameter satisfies \eqref{eq:diameter} with high probability. This in particular implies the existence of such surfaces when $n \to \infty$ and hence the claim.\end{proof}

Before we get to the proof of this proposition, we describe the basic structure of it. The idea is to explore the neighborhood of a random vertex for \emph{the hyperbolic metric on $S_{a,n}$} until we find $\approx \sqrt{n}$ vertices. We then show as in \cite{BFdlV} that the neighborhood explored is almost tree-like with only a few defects. In turns, the volume growth (in the hyperbolic metric) around such a point is the same as in the pants tree $T_a$, i.e. the radius reached is of order $ \frac{1}{\delta_a} \log  \sqrt{n}$. If we perform two such explorations from typical vertices then those explorations will have merged  and thus the distance between those points is less than $2 \times  \frac{1}{\delta_a} \log  \sqrt{n} = \frac{1}{\delta_a} \log n$. Since those bounds holds with very high probability, they hold for any pair of points instead of just typical points.

\begin{proof}[Proof of the Proposition.] Start from a given vertex of the random trivalent graph $ \mathcal{G}_{n}$ from which our surface is built, and let us explore iteratively its neighboring vertices (they correspond to the midpoints of the associated pairs of pants) using the hyperbolic distance inside $S_{a,n}$. More precisely, remember that $ \mathcal{G}_{n}$ can be built by paring in a uniform fashion the legs of $2n$ vertices, each of them having $3$ legs. We shall thus start from a given vertex $\rho$ and pair step by step the legs to grow the neighborhood of the vertex $\rho$. Iteratively at step $i\geq0$, if $\partial \mathcal{E}_i$ denotes the set of legs in the component of the origin which have not been paired yet, then we decide to pair a leg of $ \partial \mathcal{E}_{i}$ of $\rho$ whose corresponding segment in $T_a$ \emph{minimizes the hyperbolic distance in $T_a$ to the midpoint of the face where we started}.  Two events may occur at step $i$: 
\begin{itemize}
\item either we discover a new vertex of $ \mathcal{G}_{n}$ (i.e.~a new pair of pants of $S_{a,n}$) in which case the explored component $ \mathcal{E}_{i+1}$ gains a pair of pants and one leg on its boundary (it gains $2$ and loses $1$),
\item  or the edge is paired with another leg on the boundary of the explored part $\partial \mathcal{E}_{i}$. In this case, $ \partial \mathcal{E}_{i+1}$ has lost two legs compared to $ \partial \mathcal{E}_{i}$ and we put a cross (a defect) on the two corresponding segments of $T_a$, see Figure \ref{fig:explo} for an illustration. These steps are called \emph{bad steps}.
\end{itemize}

\begin{figure}[!h]
 \begin{center}
 \begin{overpic}{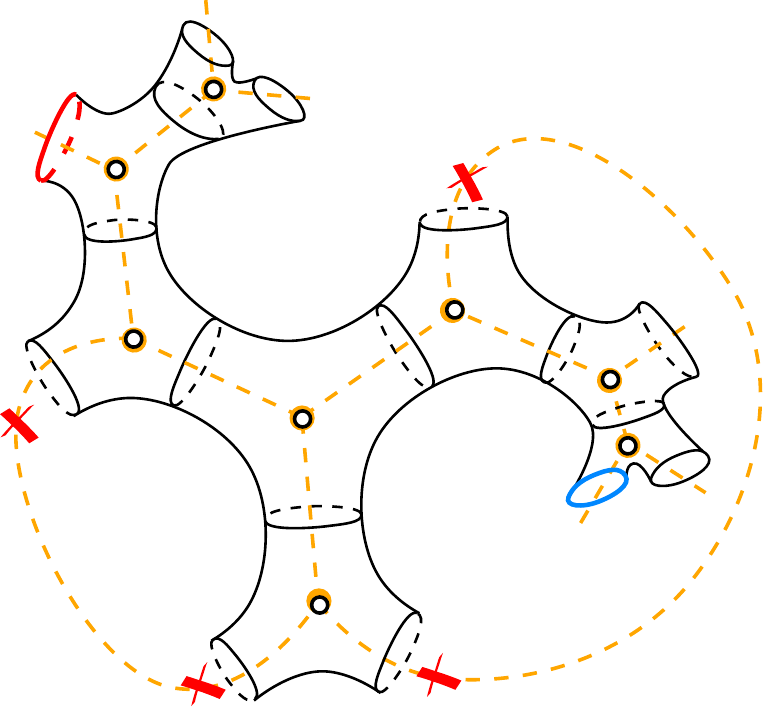}
  \put(57,38) {$1$}
  \put(28,52) {$2$} 
  \put(31,23) {$3$} 
  \put(52,65) {$4$} 
  \put(21.5,62) {$5$}    
  \put(26,11) {$6$}   
  \put(69.5,37.5) {$7$}    
  \put(27,69) {$8$}
  \put(74,34) {$9$}      
  \put(69,26) {$10$}        
\end{overpic}
 \caption{ \label{fig:explo} Illustration of the exploration of the neighborhood of a vertex in $ \mathcal{G}_{n}$ using the hyperbolic metric on $T_a$. The numbers display the order in which the half-edges are explored. In particular notice that at step $10$, we try to explore an edge which is farther from the origin in terms of the graph distance on the tree compared to the red edge (but which must be closer in terms of the hyperbolic distance). The crossed edges are half-edges which are paired during the exploration with already explored pants.}
 \end{center}
 \end{figure}
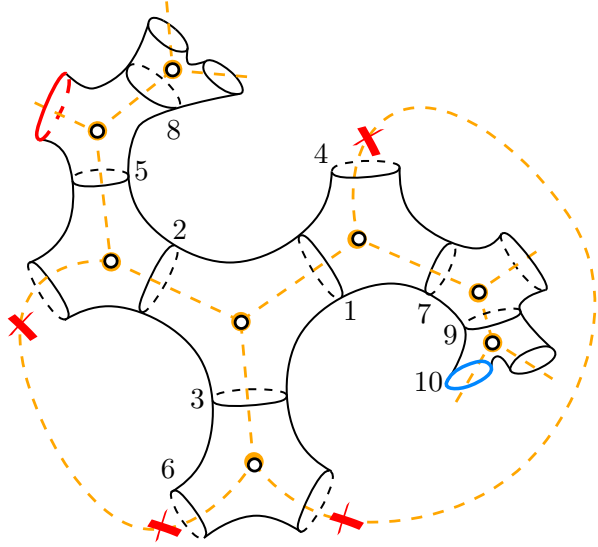
 
 We shall explore the neighborhood of $\rho$ in this fashion until the time $\tau$ where we have found $n^{1/2} \log n$ vertices (equivalently pairs of pants) in $ \mathcal{E}_{i}$. It could be that the exploration finishes before time $\tau$ if we pair all the legs i.e. if $\partial \mathcal{E}_{i} = \varnothing$. We shall call this event $ \mathcal{D}$, for ``disconnection''. This happens with a probability of order $1/n^{2}$ e.g.~by discovering a new hexagon at step $1$ and then pairing up the $4$ boundary legs of $ \partial \mathcal{E}_{1}$ in steps $2$ and $3$.

 Since $ \#\partial \mathcal{E}_{i} \leq 2+i$ and $\tau \leq 3n^{1/2}\log n$, and since the legs of $ \mathcal{G}_{n}$ are paired uniformly at random, the probability of getting a bad step $i \leq \tau$ satisfies  $$  \mathbb{P}( \mbox{step }i \mbox{ is bad}) = \frac{   \# \partial \mathcal{E}_{i}-1}{6n-2i-1} \leq  \mathrm{cst} \frac{i}{n}.$$  See Proposition 11 in \cite{BCP19} for a proof of this, since our exploration is ``Markovian''. In particular, the probability to make $k$ bad steps during the first $n^{1/2- \varepsilon}$ steps is bounded above by
\[    \frac{1}{k!}\left( n^{1/2-\eps} \right)^k \times \left( \mathrm{cst} n^{-1/2-\eps} \right)^k =  \mathrm{cst'}_{k} \cdot n^{-2 \eps k} = o(n^{-3})\]
if we choose $k>3/\eps$ (note that $k$ is a constant depending only on $\eps$). By a similar calculation, the probability to perform $\log^3 n$ bad steps inbetween time $n^{{1/2- \eps}}$ and $\tau$ is bounded above (for large $n$'s) by 
\[\frac{1}{(\log^3(n))!} \left(3n^{1/2} \log(n)\right)^{\log^3(n)} \times \left(\mathrm{cst} \log(n) n^{-1/2} \right)^{\log^3(n)} \underset{ \mathrm{Stirling}}{\leq} \left(\frac{\mathrm{cst}'}{\log(n)}\right)^{\log^3(n)} = o(n^{-3}). \]

Let us now present a deterministic lemma:
\begin{lemma} For any $ \varepsilon>0$ and any $a \in (0, \infty)$, suppose that during an exploration as above, there are fewer than $k$ bad steps until time $n^{{1/2 - \varepsilon}}$ and less than $ \log^{3}n$ bad steps until time $\tau$. Then the maximal distance $R$ reached in $T_a$ (hence in $S_{a,n}$) from the midpoint of the face where the exploration started satisfies asymptotically
$$ R \leq  \frac{1}{2}\left( \frac{1}{\delta_a}+ \varepsilon \right) \log n.$$
\end{lemma}

\begin{proof} If there were no bad steps during the exploration, we would have explored a disk inside $T_a$ until we have found $n^{1/2}\log n$ midpoints of faces. By Theorem \ref{thm:counting} and because the diameter of the pairs of pants is bounded by a constant depending on $a$ only we would have asymptotically 
$$ R \leq  \frac{1}{\delta_a}\log \left(  \mathrm{cst}_a \cdot n^{{1/2}} \log n\right).$$
So, we need to argue that in a situation where we allow the defects as above, there is still exponential volume growth at the same rate.

First consider the first $n^{1/2-\varepsilon}$ steps (the first phase) of the exploration. Assuming that no disconnection occurs before graph-distance $k$ is reached, we note that there must be at least one vertex $\eta$ at graph distance at most $k$ from $\rho$ none of whose descendents obtain a defect in the first phase of the exploration. Let us lift this situation to $T_a$, so that $0$ is a lift of $\rho$. Set $r = \mathrm{d}(0,\eta)$ and note that by symmetry\footnote{The $\frac{1}{3}$ comes from the fact that $\eta$ itself needs to be counted once and not for $\frac{2}{3}$.}
\begin{multline*}
\#\left\{ m \text{ a midpoint of a pair of pants in }T_a; \begin{array}{c}
d(m,\eta) \leq t-r \\
m \text{ lies beyond } \eta
\end{array} \right\}\\
 = \frac{1}{3}+ \frac{2}{3}\#\left\{ m \text{ a midpoint of a pair of pants in }T_a; \;
d(m,\eta) \leq t-r  \right\} =  \frac{1}{3}+ \frac{2}{3}N_a(t-r)
\end{multline*}
for all $t \geq 0$, where the phrase ``$m$ lies beyond $\eta$'' is shorthand for ``the geodesic between $O$ and $m$ intersects the pair of pants corresponding to $\eta$''. Since all these midpoints can be reached by the exploration and $r$ is uniformly bounded (in terms of $k$ and $a$), the argument above tells us that the distance $R_1$ reached after the first phase of the exploration satisfies
\[R_1 \leq \frac{1}{\delta_a} \log\left(n^{1/2-\varepsilon} \cdot \frac{3}{2 \mathrm{cst}_a}\right).\]

For the second phase of the exploration we note that at the end of the first phase, we have discovered at least $n^{1/2-2\varepsilon}$ midpoints at distance at least $R_1-\mathrm{cst}$, none of which lie beyond one another. Since there are fewer than $\log^3(n)$ defects in the second phase, beyond at least $n^{1/2-3\varepsilon}$ of these midpoints, there is no defect in the second phase either. So with the same argument as above, we obtain the lemma.
\end{proof}

 Let us now finish the proof of the proposition. Gathering-up our findings, we have seen that the exploration of the ``hyperbolic neighborhood'' of size $n^{{1/2}}\log n$ of a vertex $\rho$ in $S_{a,n}$ either disconnects the surface, or has radius bounded by $  \frac{1}{2}\left( \frac{1}{\delta_a}+ \varepsilon \right) \log n$ with a probability $1 - o(n^{{-3}})$. We claim that if we explore as above the neighborhood of another uniformly chosen vertex $\rho'$, then we either disconnect the surface or merge with the previous exploration with very high probability. Indeed, since after the first exploration we have roughly $ n^{1/2} \log n$ half-edges available, the probability that second exploration runs until time $\tau$ and avoids linking to them is bounded above by 
 $$ \left( 1-\frac{\log n}{ \sqrt{n}}\right)^{  \sqrt{n} \log n}  \leq \exp( - \log^{2} n) = o(n^{-3}).$$
 Performing a union bound over the $n^{2}$ pair of vertices of $ \mathcal{G}_{n}$ we indeed deduce that with probability at least $1-o(n^{-1})$ we have the dichotomy:
 $$ \mbox{either } S_{a, n} \mbox{ is disconnected } \quad \mbox{ or } \quad \sup_{m_{f},m_{f'}} d_{ \mathrm{hyp}}(m_{f},m_{f'}) \leq 2 \times \frac{1}{2}\left( \frac{1}{\delta_a}+ \varepsilon \right) \log n.$$ But since we know that $ \mathcal{G}_n$ or equivalently $S_{a,n}$ is connected with high probability (see e.g. \cite{BCP19})  our proposition follows. \end{proof}

\paragraph{Comments.} We end the paper with a few remarks and questions. First, another natural quantity to try to minimize as a function over moduli space is the ratio 
\[
 \mathrm{diam}(X) \big/ \mathrm{sys}(X),
\] 
where $\mathrm{sys}(X)$ denotes the systole\footnote{Recall that the systole of a closed hyperbolic surface is the length of the shortest closed geodesic on this surface.} of $X$. The analogous problem for regular graphs is well studied (see for instance \cite{ArzBis} and references therein). The open problem of finding the maximal systole of a closed hyperbolic surface of genus $g$ has also received a lot of attention \cite{BuserSarnak, Parlier, Schmutz, KatzSabourau, KSV, PW, Pet, Hamenstaedt, FortierBourqueRafi}.

One could also try to find the second order in our Theorem \ref{thm_main}. In the case of random graphs, this is a $\log \log$ \cite{BFdlV}. In our setting, this would require to let $a \to \infty$ at the right speed when $n \to \infty$. A heuristic argument, based in part on \cite[Theorem 3.5]{McMullen2}, makes us believe that an error term of the order $\log \log (\mathrm{genus})$ might be attainable.

Finally, it would also be interesting to have a more “explicit” construction of a sequence of closed hyperbolic surfaces whose diameters are asymptotic to $\log(g)$. For instance, can this be done with a sequence of
congruence covers of a closed arithmetic surface?

\bibliographystyle{alpha}
\bibliography{bib_onelogg.bib}

\begin{bibdiv}
\begin{biblist}

\bib{ArzBis}{article}{
      author={{Arzhantseva}, Goulnara},
      author={{Biswas}, Arindam},
       title={{Large girth graphs with bounded diameter-by-girth ratio}},
        date={2018Mar},
     journal={arXiv e-prints},
       pages={arXiv:1803.09229},
      eprint={1803.09229},
}

\bib{Bavard}{article}{
      author={Bavard, Christophe},
       title={{Disques extr\'{e}maux et surfaces modulaires}},
        date={1996},
     journal={Annales de la facult\'{e} des sciences de Toulouse},
      volume={5},
      number={2},
       pages={191\ndash 202},
}

\bib{BFdlV}{article}{
      author={Bollob\'{a}s, B.},
      author={Fernandez de~la Vega, W.},
       title={The diameter of random regular graphs},
        date={1982},
        ISSN={0209-9683},
     journal={Combinatorica},
      volume={2},
      number={2},
       pages={125\ndash 134},
         url={https://doi.org/10.1007/BF02579310},
      review={\MR{685038}},
}

\bib{BrooksNumberTheory}{incollection}{
      author={Brooks, Robert},
       title={Some relations between spectral geometry and number theory},
        date={1992},
   booktitle={Topology '90 ({C}olumbus, {OH}, 1990)},
      series={Ohio State Univ. Math. Res. Inst. Publ.},
      volume={1},
   publisher={de Gruyter, Berlin},
       pages={61\ndash 75},
      review={\MR{1184403}},
}

\bib{Brooks}{article}{
      author={Brooks, Robert},
       title={Platonic surfaces},
        date={1999},
        ISSN={0010-2571},
     journal={Comment. Math. Helv.},
      volume={74},
      number={1},
       pages={156\ndash 170},
         url={https://doi.org/10.1007/s000140050082},
      review={\MR{1677565}},
}

\bib{BM04}{article}{
      author={Brooks, Robert},
      author={Makover, Eran},
       title={Random construction of {R}iemann surfaces},
        date={2004},
        ISSN={0022-040X},
     journal={J. Differential Geom.},
      volume={68},
      number={1},
       pages={121\ndash 157},
         url={http://projecteuclid.org/euclid.jdg/1102536712},
      review={\MR{2152911}},
}

\bib{BCP19+}{unpublished}{
      author={Budzinski, T.},
      author={Curien, N.},
      author={Petri, B.},
       title={The diameter of {B}rooks-{M}akover random surfaces},
        note={In preparation},
}

\bib{BCP19}{article}{
      author={Budzinski, Thomas},
      author={Curien, Nicolas},
      author={Petri, Bram},
       title={Universality for random surfaces in unconstrained genus},
        date={to appear},
     journal={Electr. J. Combinatorics},
}

\bib{BuserSarnak}{article}{
      author={Buser, P.},
      author={Sarnak, P.},
       title={On the period matrix of a {R}iemann surface of large genus},
        date={1994},
        ISSN={0020-9910},
     journal={Invent. Math.},
      volume={117},
      number={1},
       pages={27\ndash 56},
         url={https://doi.org/10.1007/BF01232233},
        note={With an appendix by J. H. Conway and N. J. A. Sloane},
      review={\MR{1269424}},
}

\bib{Clozel}{article}{
      author={Clozel, Laurent},
       title={D\'{e}monstration de la conjecture {$\tau$}},
        date={2003},
        ISSN={0020-9910},
     journal={Invent. Math.},
      volume={151},
      number={2},
       pages={297\ndash 328},
         url={https://doi.org/10.1007/s00222-002-0253-8},
      review={\MR{1953260}},
}

\bib{FortierBourqueRafi}{unpublished}{
      author={Fortier-Bourque, Maxime},
      author={Rafi, Kasra},
       title={Local maxima of the systole function},
        date={2018},
        note={Preprint, arXiv: 1807.08367},
}

\bib{GorodnikNevo}{article}{
      author={Gorodnik, Alexander},
      author={Nevo, Amos},
       title={Counting lattice points},
        date={2012},
        ISSN={0075-4102},
     journal={J. Reine Angew. Math.},
      volume={663},
       pages={127\ndash 176},
         url={https://doi.org/10.1515/CRELLE.2011.096},
      review={\MR{2889708}},
}

\bib{Hamenstaedt}{article}{
      author={Hamenst\"{a}dt, Ursula},
       title={New examples of maximal surfaces},
        date={2001},
        ISSN={0013-8584},
     journal={Enseign. Math. (2)},
      volume={47},
      number={1-2},
       pages={65\ndash 101},
      review={\MR{1844895}},
}

\bib{KatzSabourau}{article}{
      author={Katz, Mikhail~G.},
      author={Sabourau, St\'{e}phane},
       title={Entropy of systolically extremal surfaces and asymptotic bounds},
        date={2005},
        ISSN={0143-3857},
     journal={Ergodic Theory Dynam. Systems},
      volume={25},
      number={4},
       pages={1209\ndash 1220},
         url={https://doi.org/10.1017/S0143385704001014},
      review={\MR{2158402}},
}

\bib{KSV}{article}{
      author={Katz, Mikhail~G.},
      author={Schaps, Mary},
      author={Vishne, Uzi},
       title={Logarithmic growth of systole of arithmetic {R}iemann surfaces
  along congruence subgroups},
        date={2007},
        ISSN={0022-040X},
     journal={J. Differential Geom.},
      volume={76},
      number={3},
       pages={399\ndash 422},
         url={http://projecteuclid.org/euclid.jdg/1180135693},
      review={\MR{2331526}},
}

\bib{Ken15}{unpublished}{
      author={Kenyon, Richard},
       title={Right-angled hexagon tilings of the hyperbolic plane},
        date={2015},
        note={Preprint, arXiv:1503.05510},
}

\bib{LaxPhillips}{article}{
      author={Lax, Peter~D.},
      author={Phillips, Ralph~S.},
       title={The asymptotic distribution of lattice points in {E}uclidean and
  non-{E}uclidean spaces},
        date={1982},
        ISSN={0022-1236},
     journal={J. Functional Analysis},
      volume={46},
      number={3},
       pages={280\ndash 350},
         url={https://doi.org/10.1016/0022-1236(82)90050-7},
      review={\MR{661875}},
}

\bib{McMullen2}{article}{
      author={McMullen, Curtis~T.},
       title={Hausdorff dimension and conformal dynamics. {III}. {C}omputation
  of dimension},
        date={1998},
        ISSN={0002-9327},
     journal={Amer. J. Math.},
      volume={120},
      number={4},
       pages={691\ndash 721},
  url={http://muse.jhu.edu/journals/american_journal_of_mathematics/v120/120.4mcmullen.pdf},
      review={\MR{1637951}},
}

\bib{Mirzakhani}{article}{
      author={Mirzakhani, Maryam},
       title={Growth of {W}eil-{P}etersson volumes and random hyperbolic
  surfaces of large genus},
        date={2013},
        ISSN={0022-040X},
     journal={J. Differential Geom.},
      volume={94},
      number={2},
       pages={267\ndash 300},
         url={http://projecteuclid.org/euclid.jdg/1367438650},
      review={\MR{3080483}},
}

\bib{Parlier}{incollection}{
      author={Parlier, Hugo},
       title={Simple closed geodesics and the study of {T}eichm\"{u}ller
  spaces},
        date={2014},
   booktitle={Handbook of {T}eichm\"{u}ller theory. {V}ol. {IV}},
      series={IRMA Lect. Math. Theor. Phys.},
      volume={19},
   publisher={Eur. Math. Soc., Z\"{u}rich},
       pages={113\ndash 134},
         url={https://doi.org/10.4171/117-1/3},
      review={\MR{3289700}},
}

\bib{Patterson1}{article}{
      author={Patterson, S.~J.},
       title={The limit set of a {F}uchsian group},
        date={1976},
        ISSN={0001-5962},
     journal={Acta Math.},
      volume={136},
      number={3-4},
       pages={241\ndash 273},
         url={https://doi.org/10.1007/BF02392046},
      review={\MR{0450547}},
}

\bib{Patterson2}{article}{
      author={Patterson, S.~J.},
       title={On a lattice-point problem in hyperbolic space and related
  questions in spectral theory},
        date={1988},
        ISSN={0004-2080},
     journal={Ark. Mat.},
      volume={26},
      number={1},
       pages={167\ndash 172},
         url={https://doi.org/10.1007/BF02386116},
      review={\MR{948288}},
}

\bib{Pet}{article}{
      author={Petri, Bram},
       title={Hyperbolic surfaces with long systoles that form a pants
  decomposition},
        date={2018},
        ISSN={0002-9939},
     journal={Proc. Amer. Math. Soc.},
      volume={146},
      number={3},
       pages={1069\ndash 1081},
         url={https://doi.org/10.1090/proc/13806},
      review={\MR{3750219}},
}

\bib{PW}{article}{
      author={Petri, Bram},
      author={Walker, Alexander},
       title={Graphs of large girth and surfaces of large systole},
        date={2018},
     journal={Math. Res. Lett.},
      volume={25},
      number={6},
       pages={1937 \ndash  1956},
}

\bib{SarnakXue}{article}{
      author={Sarnak, Peter},
      author={Xue, Xiao~Xi},
       title={Bounds for multiplicities of automorphic representations},
        date={1991},
        ISSN={0012-7094},
     journal={Duke Math. J.},
      volume={64},
      number={1},
       pages={207\ndash 227},
         url={https://doi.org/10.1215/S0012-7094-91-06410-0},
      review={\MR{1131400}},
}

\bib{Schmutz}{article}{
      author={Schmutz~Schaller, Paul},
       title={Geometry of {R}iemann surfaces based on closed geodesics},
        date={1998},
        ISSN={0273-0979},
     journal={Bull. Amer. Math. Soc. (N.S.)},
      volume={35},
      number={3},
       pages={193\ndash 214},
         url={https://doi.org/10.1090/S0273-0979-98-00750-2},
      review={\MR{1609636}},
}

\bib{Selberg}{incollection}{
      author={Selberg, Atle},
       title={On the estimation of {F}ourier coefficients of modular forms},
        date={1965},
   booktitle={Proc. {S}ympos. {P}ure {M}ath., {V}ol. {VIII}},
   publisher={Amer. Math. Soc., Providence, R.I.},
       pages={1\ndash 15},
      review={\MR{0182610}},
}

\end{biblist}
\end{bibdiv}

\end{document}